\documentclass{amsart}

\usepackage[utf8]{inputenc}
\usepackage[T1]{fontenc}

\usepackage{lmodern}

\usepackage{amsmath}
\usepackage{amssymb}
\usepackage{mathtools}
\usepackage{comment}

\usepackage[backend=bibtex, style=numeric, hyperref=true, giveninits=true]{biblatex}
\renewbibmacro{in:}{%
	\ifentrytype{article}{}{\printtext{\bibstring{in}\intitlepunct}}}

\usepackage{xspace}
\usepackage{xargs}
\usepackage{tikz-cd}
\usetikzlibrary{arrows.meta}

\usepackage{hyperref}

\newtheorem{thm}{Theorem}[section]
\newtheorem{prop}[thm]{Proposition}
\newtheorem{lemma}[thm]{Lemma}
\newtheorem{coro}[thm]{Corollary}

\theoremstyle{definition}

\newtheorem{paragr}[thm]{}

\newtheorem{remark}[thm]{Remark}

\makeatletter

\def\xpoint{\futurelet\@let@token\@xpoint}
\def\@xpoint{%
	\ifx\@let@token.\else
	.%
	\fi
	\@\xspace}

\newcommand{\loccit}{loc.~cit\xpoint}
\newcommand*{\ie}{i.e\xpoint}

\newcommand\forlang\emph
\newcommand\ndef\emph
\newcommand\nbd\nobreakdash

\renewcommand{\phi}{\varphi}
\renewcommand{\leq}{\leqslant}
\renewcommand{\geq}{\geqslant}
\renewcommand\le\leqslant
\renewcommand\ge\geqslant
\newcommand\N{\mathbb{N}}

\newcommand\leN{\le_\N}

\renewcommand{\epsilon}{\varepsilon}
\newcommand{\eps}{\epsilon}

\newcommand\limind\varinjlim
\newcommand\limproj\varprojlim

\newcommand\C{\mathcal{C}}
\newcommand\D{\mathcal{D}}
\newcommand{\Hom}{\operatorname{\mathsf{Hom}}}
\newcommand{\pref}[1]{{\widehat{ #1 }}}
\newcommand\id[1]{1_{#1}}

\newcommand{\SC}{\mathcal{SC}}
\newcommand{\Pos}{\mathcal{P}\mspace{-2.mu}\it{os}}
\newcommand{\Cat}{{\mathcal{C}\mspace{-2.mu}\it{at}}}
\newcommand{\nCat}[1]{{#1}\hbox{\protect\nbd-}\kern1pt\Cat}
\newcommand{\oonCat}[1]{(\infty, #1)\hbox{\protect\nbd-}\kern1pt\Cat}
\newcommand{\ooCat}{\nCat{\infty}}
\newcommand{\ti}[1]{\tau_{\leq #1}}
\newcommand\oo{$\infty$\nbd}
\newcommand\comp\ast

\newcommand\cDelta{\mathbf{\Delta}}
\newcommand\Deltan[1]{\varDelta_{#1}}
\newcommand\EnsSimp{\pref{\cDelta}}
\newcommand{\EnsSSimp}{\pref{\cDelta'}}
\newcommand\Sd{\mathrm{Sd}}

\newcommand\cC{\mathsf{c}}
\newcommand\nN{\mathsf{N}}
\newcommand\On[1]{\mathcal{O}_{#1}}
\newcommand\Onm[2]{\mathcal{O}_{#1}^{\leq #2}}
\newcommand\etan[2]{\eta_{#1}^{\leq #2}}
\newcommand\epsn[2]{\eps_{#1}^{\le #2}}

\newcommand{\Cda}{\mathcal{C}_{\mathrm{da}}}
\newcommand{\nCda}[1]{{#1}\hbox{\protect\nbd-}\kern1pt\Cda}
\newcommand\Z{\mathbb{Z}}
\newcommand{\atom}[1]{\langle{#1}\rangle}
\newcommand{\tabld}[2]{\begin{pmatrix}#1^0_0 &\dots &#1^0_{#2-1}
  &#1^0_{#2}\cr\noalign{\vskip 3pt} #1^1_0 &\dots &#1^1_{#2-1}
  &#1^1_{#2}\end{pmatrix}}
\newcommand{\tablnu}[3]{\begin{pmatrix}#1^0_0 &\dots &#1^0_{#2-1}
  &#3^0_n\cr\noalign{\vskip 3pt} #1^1_0 &\dots &#1^1_{#2-1} &#3^1_n\end{pmatrix}}
\newcommand{\tableq}[3]{\begin{pmatrix}#1^0_0 &\dots &#1^0_{#2-1}
		&\overline{#3^0_n}\cr\noalign{\vskip 3pt} #1^1_0 &\dots &#1^1_{#2-1} &\overline{#3^1_n}\end{pmatrix}}
\newcommand{\tabll}[2]{\begin{pmatrix}#1^0_0 &#1^0_1 &\dots &#1^0_{#2-1}
  &#1^0_{#2}\cr\noalign{\vskip 3pt} #1^1_0 &#1^1_1 &\dots &#1^1_{#2-1}
  &#1^1_{#2}\end{pmatrix}}

\newcommand{\supp}{\operatorname{supp}}

\newcommand\pdfn{\texorpdfstring{$n$}{n}}

\makeatother

\author{Andrea Gagna}
\address{%
	Aix Marseille Univ\\
	CNRS\\
	Centrale Marseille\\
	I2M\\
	Marseille\\
	France}
\email{andrea.gagna@gmail.com}

\title[$\nCat{n}$ and $\Cda$ model homotopy types]
  {Strict \pdfn-categories and augmented directed complexes model homotopy types}

\begin{document}

\begin{abstract}
	In this paper we show that both the homotopy category of strict $n$\nbd-categ\-ories, $1\leqslant n \leqslant \infty$,
	and the homotopy category of Steiner's augmented directed
	complexes are equivalent to the category of homotopy types.
	In order to do so, we first prove a general result, based on a strategy of Fritsch and Latch,
	giving sufficient conditions
	for a nerve functor with values in simplicial sets to induce an equivalence at the level of homotopy categories.
	We then apply this result to strict $n$-categories and augmented directed complexes, where the assumption of our theorem
	were first established by Ara and Maltsiniotis and of which we give an independent proof.
\end{abstract}

\subjclass{18D05, 18G30, 18G35, 18G55, 55P10, 55P15, 55U10, 55U15, 55U35}

\keywords{strict $\infty$-categories, augmented directed complexes, simplicial sets, homotopy types, Street's nerve, orientals}

\maketitle

\section*{Introduction}

The homotopy theory of small categories was born
with the introduction by Grothendieck
of the nerve functor in~\cite{grothendieck_techniques_III},
allowing the definition of weak equivalences in
$\Cat$, the category of small categories: 
a functor is a weak equivalence precisely
when its nerve
is a simplicial weak equivalence.

The first striking result of this theory
appears in Illusie's thesis~\cite{cotangent} (who credits it
to Quillen) and states that the induced
nerve functor induces an equivalence at the level of the homotopy
categories.
A careful study of the subtle homotopy theory
of small categories by 
Thomason~\cite{Cat_closed} lead him to show
another important result: the existence
of a model structure on $\Cat$ which is Quillen equivalent
to the Kan--Quillen model structure
on simplicial sets. This implies that
small categories and simplicial sets are not
only equivalent as homotopy categories, but
actually as $(\infty, 1)$-categories.

Similarly Chiche~\cite{chiche_homotopy},
using ideas of del~Hoyo~\cite{del-Hoyo},
has showed
that strict $2$\nbd-categ\-ories model homotopy types.
Moreover, Ara and Maltsiniotis~\cite{nThomason} provided an abstract framework in which
a Quillen equivalence between strict $n$-categories and simplicial sets can be
established, using the Street nerve, as long as some conditions are satisfied;
they also showed that these conditions hold
for small $2$-categories.
In a subsequent paper~\cite{e} they proved in full generality one
of these conditions, namely condition~(e), making an important use
of augmented directed complexes introduced by Steiner~\cite{Steiner1},
which act as a linearisation of $\infty$-categories.

In the same article, they also conjectured that for every $1\leqslant n \leqslant \infty$,
the Street nerve induces an equivalence of homotopy categories
between strict $n$-categories and simplicial sets and that the same
holds for augmented directed complexes with their own nerve functor.
The aim of this paper is to prove these conjectures.

We consider a general framework with a nerve $U\colon \C \to \EnsSimp$
induced as a right adjoint of a left Kan extension $F$ of
a cosimplicial object of $\C$ along the Yoneda embedding of $\cDelta$,
where $\C$ is a cocomplete category.
Using the properties of simplicial subdivision and adapting a strategy
developed in an article by Fritsch and Latch~\cite{FritschLatch},
a short proof that $U$ induces
an equivalence at the level of homotopy categories can be given, as long as
it satisfies an appropriate version of condition~(e).
This condition says that the unit $\eta\colon 1_{\EnsSimp} \to U F$
is a weak equivalence on the nerves of posets,
that is for every poset $E$, the morphism of
simplicial sets $N_1(E) \to UF(N_1(E))$ is a weak equivalence,
where $N_1\colon \Cat \to \EnsSimp$ denotes the usual nerve.
As a first example, we give a short proof of the well\nbd-known
result stating that (ordered) simplicial complexes model homotopy types.
More interestingly, Ara and Maltsiniotis~\cite{e} have established condition~(e) for both
the nerve $N_n\colon \nCat{n} \to \EnsSimp$
of strict $n$\nbd-categories, $1\leqslant n\leqslant \infty$,
and the nerve $\nN \colon \Cda \to \EnsSimp$ of augmented directed complexes.
We can thus apply our theorem to these cases and get a proof of
the above conjectures. Nevertheless, we give an independent
proof of the condition for the cases of need which is more
simplicial in spirit, while Ara and Maltsiniotis' one is
purely \oo-categorical. In doing so, we establish some
results about truncations of orientals which we use to
extend a result of Steiner~\cite{Steiner1}: we show that a certain subcategory
of the augmented directed complexes is equivalent
to a subcategory of small \oo-categories containing the orientals
and their truncations.
This is particularly useful to define and deal with
most of the \oo-functors between \oo-categories that we will need.

Let us fix some notations and terminology.
If $A$ is a small category, the category of presheaves on $A$
will be denoted by $\pref{A}$.
If $F$ and $G$ are two functors between two categories $\C$ to $\D$
and $W$ is a class of arrows of $\D$ called weak equivalences,
we will say that a natural transformation $\alpha \colon F \to G$ is a
\ndef{weak equivalence} when it is a component-wise weak equivalence,
\ie for any object $c$ of $\C$ the arrow $\alpha_c$ is an
element of $W$. For any adjunction, we will denote by $\eta$ the unit
and $\epsilon$ the counit.

\section*{Acknowledgements}

The author would like to thank Dimitri Ara for his guidance and the
many helpful improvements he suggested  during the development of this paper.

\section{Simplicial preliminaries}

\begin{paragr}
	The \ndef{category of simplices} $\cDelta$ is the category whose
	objects are the non-empty finite ordinals $\Delta_n =\{0, 1, \dots, n\}$,
	for $n\geqslant 0$, and whose arrows are the order-preserving
	maps; the \ndef{category of semi-simplices} $\cDelta'$ is
	the category with the same objects as
	$\cDelta$ but only strictly increasing maps.
	We have a canonical inclusion functor
	$i \colon \cDelta' \to \cDelta$.
	We will denote by $i^*\colon \EnsSimp \to \EnsSSimp$
	the functor given by precomposition by $i$ and by
	$i_!\colon \EnsSSimp \to \EnsSimp$ its left adjoint.
	
	We denote by $\Cat$ the category of small categories and
	by $N_1\colon \Cat \to \EnsSimp$ and $c_1\colon \EnsSimp \to \Cat$
	the usual nerve functor and its left adjoint.
\end{paragr}

\begin{paragr}
 We shall call \ndef{weak equivalences of} $\EnsSimp$
 the usual weak homotopy equivalences and
 \ndef{weak equivalences of} $\EnsSSimp$
 the morphisms such that their
 image by $i_!$ is a weak equivalence of $\EnsSimp$.
\end{paragr}

\begin{prop}
 The adjunction morphisms $\eta \colon 1_{\EnsSSimp} \to i^*i_!$
 and $\epsilon\colon i_!i^* \to 1_{\EnsSimp}$ are weak equivalences.
\end{prop}

\begin{proof}
 See Proposition~2.12 of~\cite{FritschLatch}.
\end{proof}

\begin{coro}
 The functors $i_!$ and $i^*$ both respect weak equivalences
 and induce an equivalence pair at the level of homotopy categories.
\end{coro}

\begin{proof}
 This follows immediately from the previous proposition.
\end{proof}

\begin{paragr}\label{paragr:fritsch-latch_functor}
 We denote by $\Sd\colon \EnsSimp \to \EnsSimp$ the subdivision endofunctor of simplicial sets
 introduced by Kan in~\cite[§7]{Kan} and by $\alpha\colon \Sd \to 1_{\EnsSimp}$
 the canonical natural transformation. Recall that $\alpha$
 is a weak equivalence (see~Lemma~7.5 of~\loccit).
 Therefore the previous proposition implies that the natural transformation
 \[\alpha^2\ast\epsilon \colon \Sd^2i_!i^* \to 1_{\EnsSimp}\]
 is a weak equivalence, where $\alpha^2 = \alpha  \ast \alpha$.
\end{paragr}

\begin{prop}\label{prop:poset}
 For any simplicial set $X$, the simplicial set
 $\Sd^2 i_!i^*(X)$ is isomorphic to the nerve of
 a poset.
\end{prop}

\begin{proof}
 By Corollary~3.11 and Remark~3.15 of~\cite{FritschLatch},
 the simplicial set $\Sd^2 i_!i^*(X)$ is isomorphic
 to the nerve of a small category, say $N_1(A)$.
 It is in fact isomorphic to the nerve of
 the category $c_1\Sd^2 i_!i^*(X)$, as we have
 \[
 \Sd^2 i_!i^*(X) \cong N_1(A) \cong N_1c_1 N_1 (A) \cong N_1 c_1\Sd^2 i_!i^*(X).
 \]
 But, for any simplicial set $Y$, the category
 $c_1 \Sd^2(Y)$ is a poset
 (see~Lemma~5.6 of~\cite{Cat_closed}).
\end{proof}

\section{An abstract Illusie-Quillen theorem}

\begin{paragr}
We fix $\C$ a cocomplete category and $j\colon \cDelta \to \C$ a cosimplicial object of $\C$.
We denote by $U$ the induced nerve functor, right adjoint to 
the left Kan extension $F \colon \EnsSimp \to \C$ of $j$ along the Yoneda embedding of $\cDelta$.
We call \ndef{weak equivalences of $\C$} the morphisms sent by $U$ to weak
equivalences of~$\EnsSimp$. In particular, $U$ induces a functor 
at the level of the homotopy categories.
\end{paragr}

\begin{paragr}
	We shall say that \ndef{$j$ satisfies condition (e)} if the unit $\eta\colon 1_{\EnsSimp} \to U F$
	is a weak equivalence on the nerves of posets, \ie if the morphism $N_1(E) \to UF(N_1(E))$ of
	simplicial sets  is a weak equivalence for any poset $E$.
\end{paragr}

\begin{remark}
We named this property after the condition (e) appearing in the abstract
Thomason Theorem (Theorem~4.11 of~\cite{nThomason}) of Ara and Maltsiniotis.
\end{remark}

\begin{thm}[Abstract Illusie--Quillen Theorem]\label{thm:illusie-quillen}
Let $j\colon \cDelta \to \C$ be a cosimplicial object satisfying condition (e).
Then the functor $U\colon \C \to \EnsSimp$ induces an equivalence
at the level of the homotopy categories.
\end{thm}

\begin{proof}
Let $Q\colon\EnsSimp \to \EnsSimp$ be a functor
such that for every simplicial set $X$, the simplicial set
$Q(X)$ is isomorphic to the nerve of a poset
and $\gamma \colon Q \to 1_\EnsSimp$ a natural transformation
which is a weak equivalence.
For instance, we can choose the functor $\Sd^2 i_! i^*$  as $Q$
by virtue of Proposition~\ref{prop:poset} and as $\gamma$
the natural transformation $\Sd^2 i_!i^* \to 1_\EnsSimp$ considered
in paragraph~\ref{paragr:fritsch-latch_functor}.

We will show that $F Q$ respects the weak equivalences and gives an inverse in
homotopy of the functor induced by $U$.
We have a zig-zag of natural transformations
\[
  \begin{tikzcd}
	U F Q &
	Q \ar[l, "\eta\ast Q"'] \ar[r, "\gamma"] &
	1_{\EnsSimp} .
  \end{tikzcd}
\]

Condition (e) and the fact that $Q$ takes values in nerves of posets
show that $\eta\ast Q$ is a weak equivalence;
by definition, $\gamma$ is so too.
This establishes that $FQ$ preserves weak equivalences and that
it is a left inverse for $U$ in homotopy.

Moreover, we have a natural transformation
\[
 \begin{tikzcd}
  F Q U \ar[rr, "F\ast \gamma \ast U"] &&
  F U \ar[r, "\epsilon"] & 1_{\C}
 \end{tikzcd} .
\]

We have to show that
\[
 \begin{tikzcd}
  U F Q U \ar[rr, "UF\ast \gamma \ast U"] &&
  U F U \ar[r, "U\ast \epsilon"] & U
 \end{tikzcd}
\]
is a weak equivalence. Considering the commutative diagram
\[
 \begin{tikzcd}
  U F Q U \ar[rr, "UF\ast \gamma \ast U"] &&
  U F U \ar[r, "U\ast\epsilon"] & U\\
  Q U \ar[u, "\eta \ast QU"]\ar[rr, "\gamma\ast U"'] &&
  U \ar[u, "\eta\ast U"] \ar[ur, "1_{U}"']
 \end{tikzcd}\ ,
\]
we note that $\eta\ast QU$ is a weak equivalence for the same
reason as above, that is by condition (e), and that $\gamma \ast U$ is a
weak equivalence by definition. By virtue of the 2\nbd-out\nbd-of\nbd-3 for the weak
equivalences of $\C$, this proves the theorem.
\end{proof}

As a first application of the above theorem, we give a
short proof of the classical result stating that simplicial
complexes model homotopy types. The next sections
will be devoted to formulate and apply the
previous theorem to strict $n$\nbd-categories
and directed augmented complexes.

\begin{paragr}
	A simplicial complex is a pair $(E, \Phi)$,
	where $E$ is a poset and $\Phi$ is a set of non-empty linearly ordered
	finite subsets of $E$
	such that, for any $x$ in $E$, $\{x\}$ is in $\Phi$ and
	if $S$ is in $\Phi$, then for any non-empty $S'\subset S$
	we have that $S'$ is in $\Phi$ too.
	A morphism of simplicial complexes $f \colon (E, \Phi) \to (F,\Psi)$
	is a non-decreasing application $f\colon E \to F$ such that,
	for any $S$ in $\Phi$, $f(S)$ is in $\Psi$.
	
	Any poset $E$ can be see canonically as the simplicial complex
	$(E, \xi E)$, where $\xi E$ is the set of non-empty finite subsets of $E$.
	There is then a canonical embedding functor of the category of posets
	into the category of simplicial complexes. In particular, we get a
	canonical cosimplicial object $\kappa \colon \cDelta \to \SC$.
\end{paragr}

\begin{paragr}\label{paragr:SC_cocomplete}
	The category $\SC$ of simplicial complexes is cocomplete.
	Indeed, if $F\colon I \to \SC$ is a functor indexed by a small
	category $I$ and we denote by $(E_i, \Phi_i)$ the
	simplicial complex $F(i)$, for $i$ an object of $I$,
	then the colimit of $F$ is the simplicial complex $(E, \Phi)$,
	where $E$ is the colimit $\limind_I E_i$ computed in the
	category of posets and a subset $S$ of $E$ is in $\Phi$
	if and only if there exists an object $i$ of $I$
	and $S_i$ in $\Phi_i$ such that the image of $S_i$
	by the canonical morphism $E_i \to E$ is precisely $S$.	
	
	In particular, the cosimplicial object $\kappa \colon \cDelta \to \SC$,
	mapping a simplex $\Deltan{n}$ to the simplicial complex $(\Deltan{n}, \xi \Deltan{n})$,
	induces by Kan extension an adjoint pair
	\[
	 \kappa_! \colon \EnsSimp \to \SC\quad\text{and}\quad
	 \kappa^* \colon \SC \to \EnsSimp .
	\]
	
	We define the \emph{weak equivalences of $\SC$} to
	be the morphisms of simplicial complexes whose image
	by $\kappa^*$ is a simplicial weak equivalence.
\end{paragr}

\begin{remark}\label{remark:posets_as_simplicial_complexes}
	For any poset $E$, notice that the simplicial set
	$k^*(E, \xi E)$ is isomorphic to $N_1(E)$ the nerve of $E$
	(cf. paragraph~8.4 of~\cite{e}). By the triangular identities
	of the adjoint pair $(k_!, k^*)$ and the 2\nbd-out\nbd-of\nbd-3,
	in order to establish condition (e)
	for the cosimplicial object $\kappa \colon \cDelta \to \SC$
	it is enough to show that the counit morphism
	$\kappa_!\kappa^* \to \id{\SC}$ is a weak equivalence when
	evaluated to posets, \ie simplicial complexes of the form~$(E, \xi E)$.
\end{remark}

\begin{thm}
	The functor $\kappa^* \colon \SC \to \EnsSimp$
	induces an equivalence at the level of homotopy categories.
	In particular, simplicial complexes model homotopy types.
\end{thm}

\begin{proof}
	We show that the counit
	morphism $\kappa_!\kappa^* \to \id{\SC}$ is an isomorphism when
	evaluated to posets.
	By Theorem~\ref{thm:illusie-quillen} and the previous remark,
	this proves the theorem.
	
	Let $E$ be a poset. The simplicial complex
	$(E', \Phi') = k_!k^*(E, \xi E)$ is the colimit
	of the canonical functor $\cDelta/k^*(E, \xi E) \to \SC$,
	where $\cDelta/k^*(E, \xi E)$ is the category of elements of the simplicial set $k^*(E, \xi E)$;
	hence it is the colimit of the canonical functor $\cDelta/N_1(E) \to \SC$.
	The poset $E'$ is computed as the colimit of the functor
	$\cDelta/N_1(E) \to \SC \to \Pos$, where the functor $\SC \to \Pos$
	is the left adjoint to the embedding $\Pos \to \SC$ and sends a simplicial complex $(F, \Psi)$
	to the poset $F$. Now, if we consider the ajoint pair
	\[
	 c_1 \colon \EnsSimp \to \Pos \quad \text{and} \quad N_1 \colon \Pos \to \EnsSimp ,
	\]
	the colimit $E'$ of the canonical functor $\cDelta/N_1(E) \to \Pos$
	is given by $c_1 N_1(E)$ and, since the nerve $N_1$ is fully faithful,
	we have the isomorphisms
	\[
	 c_1N_1(E) \cong E \cong E'.
	\]
	We are left to check that $\Phi'$ is $\xi E$. For let $S$ be a non-empty
	finite linear order subset $S$ of $E$. There exists a unique
	object $(\Deltan{p}, \phi)$ of $\cDelta/N_1(E)$ such that $\phi \colon \Deltan{p} \to E$
	is injective and its image is precisely $S$. The explicit description given
	in paragraph~\ref{paragr:SC_cocomplete}
	of the colimits of simplicial complexes shows that $\Phi'$ contains $S$
	and therefore we get $\Phi' = \xi E$, which concludes the proof of the theorem.
\end{proof}

\section{Recollections on Steiner's theory}

\begin{paragr}
 The notion of \emph{augmented directed complex} was introduced by Steiner in~\cite{Steiner1}
 and consists of a triple~$(K, K^*, e)$ where
 \[
  K = 
  \begin{tikzcd}
   \cdots \ar[r, "d_{i+1}"] &
   K_i \ar[r, "d_i"] &
   K_{i-1} \ar[r, "d_{i-1}"] &
   \cdots \ar[r, "d_1"] &
   K_0
  \end{tikzcd}
 \]
 is a (homological) chain complex of abelian groups in non-negative degrees,
 $e \colon K_0 \to \Z$ is an augmentation, that is it satisfies $ed_1 = 0$, and
 $K^* = (K^*_i)_{i\geq 0}$ is a collection of abelian monoids such that for any
 $i \geq 0$ the component $K^*_i$ is a submonoid of the abelian group $K_i$.
 Remark that we do not ask any compatibility of the submonoids with respect to
 the differentials or the augmentation.
 
 For any $i\geq 0$ we shall call \emph{$i$-chains} the elements of $K_i$ and the
 elements of $K^*_i$ shall be called \emph{positive $i$-chains}.
 In fact, we can define a preorder on $K_i$ by setting
 \[
  x \leq y\quad\mbox{if and only if}\quad y - x \in K^*_i\,,
 \]
 for which we have
 \[
  K^*_i = \{x \in K_i : x_i \geq 0\}.
 \]
 
 A morphism from~$(K, K^* , e)$
 to $(K', K'^*, e')$ of augmented directed chain complexes  is a morphism of chain complexes $f \colon K \to K'$
 which respects the positivity and the augmentation, \ie we have
 that $f_i(K^*_i)$ is a submonoid of $K'^*_i$ for any $i\geq 0$ and $e'f_0 = e$.
 We denote by $\Cda$ the category of augmented directed complexes. We shall
 often identify an object with its underlying chain complex.
\end{paragr}

\begin{paragr}\label{paragr:basis}
 A \emph{basis} of an augmented directed complex $K$ is a graded set
 $B = (B_i)_{i\geq 0}$ such that, for any $i\geq 0$, the set $B_i$
 is both a basis for the $\Z$-module $K_i$ and a set of generators
 for the submonoid $K^*_i$ of $K_i$.
 If a complex has a basis, for any $i\geq 0$ the preorder relation
 of positivity on $K_i$ defined in the previous paragraph
  is a partial order relation and the elements $B_i$
 of the basis are the minimal elements of $K_i^*\setminus \{0\}$;
 in particular, if a basis exists then it is unique. When an augmented directed complex
 has a basis, we shall say
 that the complex is \emph{with basis}.
 
 Let $K$ be an augmented directed complex with basis $B$.
 For any $i\geq 0$, an $i$\nbd-chain~$x$ can be written
 uniquely as a linear combination of elements of $B_i$
 with integral coefficients. The \emph{support} of $x$, denoted by
 $\supp(x)$,
 is the (finite) set of elements of the basis appearing in this
 linear combination with non-zero coefficient. We can
 write the $i$-chain uniquely as the difference of two
 positive $i$-chains with disjoint supports $x = x_+ - x_-$.
 We define a matrix
 \[
  \atom{x}=\tabll{\atom{x}}{i},
\]
where the elements $\atom{x}^\varepsilon_k$ are inductively defined by:
\begin{itemize}
\item $\atom{x}^0_i = x = \atom{x}^1_i$;
\smallskip

\item $\atom{x}^0_{k - 1} = d(\atom{x}^0_k)_-$ et $\atom{x}^1_{k - 1} = d(\atom{x}^1_k)_+$\,, \,for $0 < k \leq i$\,.
\end{itemize}
We say that an augmented directed complex $K$ with basis $B$ is \emph{unital}
if, for any $i\geq 0$ and any $x$ in $B_i$, we have the equality
$e(\atom{x}^0_0) = 1 = e(\atom{x}^1_0)$.
\end{paragr}

\begin{paragr}\label{paragr:Steiner}
 Let $K$ be an augmented directed complex with basis $B$.
 For $i\geq 0$, we denote by $\leq_i$ the smallest
 preorder relation on $B = \coprod_i B_i$ satisfying
 \[
 x \leq_i y \quad\text{if}\quad |x|>i, |y|>i\text{ and }\supp{(\atom{x}^1_i)}\cap \supp{(\atom{y}^0_i)} \neq 0.
 \]
 We say that the base $B$ is \emph{loop-free} if, for any $i\geq 0$,
 the preorder relation $\leq_i$ is a partial order relation.
 We call \emph{Steiner complex}
 an augmented directed complex $K$ with unital
 and loop-free basis $B$. 
\end{paragr}

We remind the definition of
the pair of adjoint functors $\lambda \colon \ooCat \to \Cda$
and $\nu \colon \Cda \to \ooCat$
introduced by Steiner in~\cite{Steiner1}.

\begin{paragr}
	Let $C$ be a small $\infty$-category. For $i\geq 0$,
	the abelian group $\lambda(C)_i$ is generated by the
	elements of the form $[x]$, where $x$ is an
	$i$\nbd-cell of $C$, and by the relations
	\[
	[x \ast_j y] = [x] + [y],
	\]
	where $x$ and $y$ are two $j$\nbd-composable $i$\nbd-cells of $C$.
	The positivity submonoid $\lambda(C)_i^\ast$ is the submonoid generated by
	the elements $[x]$, for $x$ an $i$\nbd-cell of $C$.
	For $i > 0$, the differential $d_i \colon \lambda(C)_i \to \lambda(C)_{i-1}$
	of an element $[x]$, where $x$ is an $i$\nbd-cell of $C$, is defined by
	\[
	d_i([x]) = [t(x)] - [s(x)].
	\]
	Finally the augmentation $\eps \colon \lambda(C)_0 \to \Z$
	is the unique morphism of abelian groups sending, for any
	object $x$ of $C$, the generating element $[x]$ to $1$.
	
	If $u \colon C \to D$ is an $\infty$-functor, for $i\geq 0$ we
	define the morphism $\lambda(u)_i$ 
	by sending the generating element $[x]$ of $\lambda(C)_i$ to the generating element $[u(x)]$ of $\lambda(D)_i$.
	It is easy to check that this defines a morphism of augmented directed complexes
	from $\lambda(C)$ to $\lambda(D)$ and that $\lambda$ is a functor.	
\end{paragr}

\begin{paragr}\label{paragr:atom}
	Let $K$ be an augmented directed complex. For $i \ge 0$, the $i$-cells
	of $\nu(K)$ are the matrices
	\[
	\tabld{x}{i}
	\]
	such that
	\begin{enumerate}
		\item $x^\epsilon_k$ belongs to $K^\ast_k$ for $\epsilon = 0, 1$ and $0
		\le k \le i$ ;
		\item $d(x^\epsilon_k) = x^1_{k-1} - x^0_{k-1}$ for $\epsilon = 0, 1$ and $0
		< k \le i$ ;
		\item $e(x^\epsilon_0) = 1$ for $\epsilon = 0, 1$ ;
		\item $x_i^0 = x_i^1$.
	\end{enumerate}
	We refer to~\cite[§2]{joint} for a full description of the \oo-categorical
	structures, where the same notations are used.
	
	We remark that if an augmented directed complex $K$ has unital basis,
	then for any element $x$ of the basis of $K$, the matrix $\atom{x}$
	introduced in paragraph~\ref{paragr:basis}
	defines a cell of $\nu(K)$, that we call the \emph{atom} associated to $x$.
\end{paragr}

\begin{thm}[Steiner]\label{thm:equivalence_Steiner}
 The functors
 \[
  \lambda \colon \ooCat \to \Cda, \qquad \nu \colon \Cda \to \ooCat
 \]
 form an adjoint pair and for any Steiner complex $K$ the counit
 \[
  \lambda(\nu(K)) \to K
 \]
 is an isomorphism. In particular, the functor $\nu \colon \Cda \to \ooCat$
 is fully faithful when restricted to the full subcategory of $\Cda$ spanned
 by Steiner complexes.
\end{thm}

\begin{proof}
 See~Theorem~2.11 and~Theorem~5.6 of~\cite{Steiner1}.
\end{proof}
  
\begin{paragr}\label{paragr:def_le_N}
  Let $K$ be an augmented directed complex with basis $B$.
  We shall denote by~$\leN$ the smallest preorder relation on $B$
  satisfying
  \[
    x \leN y \quad\text{if}\quad x \in \supp(d(y)_-) \text{ or }
    y \in \supp(d(x)_+),
  \]
  where we fixed by convention $d(b) = 0$ if $b$ belongs to $B_0$.
  
  We shall say that a basis $B$ is \emph{strongly loop-free}
  if the preorder relation $\leN$ is actually a partial
  order relation and
  we shall call \emph{strong Steiner complex}
 an augmented directed complex $K$ with unital
 and strongly loop-free basis $B$. 
\end{paragr}

\begin{prop}[Steiner]
Let $K$ be an augmented directed complex with basis $B$.
If the basis $B$ is strongly loop-free, then it is loop-free.
\end{prop}

\begin{proof}
  See Proposition~3.7 of~\cite{Steiner1}.
\end{proof}

\begin{paragr}\label{paragr:truncation_Cda}
 Let $K$ be an augmented directed complex.
 We shall call \emph{dimension} of $K$ the
 smallest integer $n\geq 0$ such that $K_i=0$ for all $i>n$.
 If such an $n$ does not exist, we will say that $K$ is of infinite dimension.
 For $n\geq 0$, we shall denote by $\nCda{n}$ the full subcategory of $\Cda$
 spanned by augmented directed complexes of dimension at most $n$ and we shall
 call its objects \emph{augmented directed $n$-complexes}.
 
 The embedding functor $\nCda{n} \hookrightarrow \Cda$ has a left adjoint
 that we shall denote by
   \[
    \ti{n} \colon \Cda \to \nCda{n}
  \]
 We name \emph{$n$-truncation} this functor. Explicitly, if $K$ is an augmented directed complex, we have
 \[
  \ti{n}(K)_i =
  \begin{cases}
   K_i & \mbox{if }i< n,\\
   K_n\slash d(K_{n+1}) & \mbox{if } i=n,\\
   0 & \mbox{if }i>n,
  \end{cases}
  \quad\mbox{and}\quad
  \ti{n}(K)^*_i =
  \begin{cases}
   K^*_i & \mbox{if } i<n,\\
   \overline{K^*_n} & \mbox{if }i=n,\\
   0 & \mbox{if }i>n,
  \end{cases}
 \]
 where $\overline{K^*_n}$ denotes the image of $K^*_i$ in $K_n\slash d(K_{n+1})$,
 and where the differentials and the augmentations are inherited from those of $K$
  in the obvious way.
\end{paragr}

\section{Truncated orientals for strict \pdfn-categories}

\begin{paragr}\label{paragr:truncation_ooCat}
	For any $n$ such that $1 \leqslant n \leqslant \infty$, we denote
	by $\nCat{n}$ the category of small strict $n$-categories.
	We shall consider an $n$-category $C$ as an $\infty$-category such that
	the set of $j$-cells $C_j$ only contains trivial cells for any $j>n$.
	We recall that for any positive integer $n$,
	the full inclusion $\nCat{n} \to \ooCat$ has a left adjoint, that we shall denote by
	\[
	\ti{n}\colon \ooCat \to \nCat{n}\,.
	\]
	We name \emph{$n$-truncation} this functor. Explicitly, if $C$ is an $\infty$-category, we have
	   \vskip-9pt
	\[
	  \ti{n}(C)_i = \begin{cases}
	    C_i & \text{if $i < n$,} \\
	    {C_n}\slash{\sim} & \text{if $i \geq n$,} \\
	  \end{cases}
	\]
	where $\sim$ is the smallest equivalent relation on $C_n$ verifying
	\[
	x \sim y \mbox{ if there exists a $(n+1)$\nbd-cell of $C$ from $x$ to $y$}
	\]
	and	where the compositions and identities are inherited by cells of $C$
	in the obvious way.
\end{paragr}

\begin{paragr}
  Let us fix an integer $n \geq 0$. If $C$ is an $n$-category,
  then the augmented directed complex~$\lambda(C)$ has at most
  dimension~$n$. Similarly, if $K$ is a augmented directed complex
  of dimension~$n$, it is immediate to see that $\nu(K)$ is an $n$-category.
  The functors~$\lambda$ and~$\nu$ induce functors
  \[
    \lambda \colon \nCda{n} \to \nCat{n}
    \quad\text{and}\quad
    \nu \colon \nCat{n} \to \nCda{n}
  \]
  and these form an adjoint pair. By definition, the squares
  \[
   \begin{tikzcd}
    \nCat{n} \ar[r] \ar[d, "\lambda"'] & \ooCat \ar[d, "\lambda"] \\
       \nCda{n} \ar[r] & \Cda
   \end{tikzcd}\qquad\mbox{and}\qquad
   \begin{tikzcd}
    \nCda{n} \ar[r] \ar[d, "\nu"'] & \Cda \ar[d, "\nu"] \\
       \nCat{n} \ar[r] & \ooCat
   \end{tikzcd}\mbox{,}
  \]
  where the horizontal arrows denote the inclusion functors, are commutative.
\end{paragr}

\begin{paragr}\label{prop:truncations}
	For any integer $n \geq 0$, consider the square
  \[
    \begin{tikzcd}
      \ooCat \ar[r, "{\ti{n}}"] \ar[d, "\lambda"'] & \nCat{n} \ar[d, "\lambda"] \\
      \Cda \ar[r, "{\ti{n}}"'] & \nCda{n} 
    \end{tikzcd} .
  \]
  For any small  $\infty$\nbd-category $C$,
	the abelian groups $(\ti{n}\lambda(C))_i$ and~$(\lambda \ti{n}(C))_i$
	are equal for~$i\neq n$, so let~$i=n$.
	The generating elements of~$(\ti{n}\lambda(C))_n$ are of the form
	$\overline{[x]}$, for~$x$ a $n$\nbd-cell of $C$, while the
	generating elements of~$(\lambda \ti{n}(C))_n$ are of the
	form~$[\bar{y}]$, for $\bar{y}$ a $n$\nbd-cell of $\ti{n}(C)$.
	The map sending the element $\overline{[x]}$ to $[\bar{x}]$ is well-defined and
	extends to an isomorphism $\ti{n}\lambda(C)_n \to \lambda \ti{n}(C)_n$ of abelian groups ,
	which is moreover compatible with the differential.
	Thus the square above is commutative up to this canonical natural isomorphism.
	
	Consider now the square
	 \[
	  \begin{tikzcd}
	  \Cda \ar[r, "{\ti{n}}"] \ar[d, "\nu"'] & \nCda{n} \ar[d, "\nu"] \\
	  \ooCat \ar[r, "{\ti{n}}"'] & \nCat{n} 
	  \end{tikzcd}
	 \]
	and let us give an explicit description of the canonical morphism
	\[
	 \ti{n}\nu(K) \to \nu\ti{n}(K)
	\]
	in the case where $K$ is a augmented directed complex with basis
	(cf.~Proposition~2.23 of~\cite{joint}).
	For any $i<n$ the canonical map from $(\ti{n}\nu(K))_i$ to $(\nu\ti{n}(K))_i$
	is the identity. An $n$\nbd-cell of $\ti{n}\nu(K)$
	is an equivalence class of an element
	\[
	 \tabld{x}{n}
	\]
	of $\nu(K)_n$ with respect to the equivalence relation
	identifying the previous element with another element of the form
	\[
	\tablnu{x}{n}{y}
	\]
	if and only if there exists $z$ in $K_{n+1}$ such that
	$d(z) = y^1_n - x^0_n$.
	Besides, an $n$\nbd-cell of $\nu\ti{n}(K)$
	is a matrix
	\[
	 \tablnu{x}{n}{t}
	\]
	where $x^\eps_i$ is in $K_i^\ast$, for all $0\leq i < n$ and $\eps = 0, 1$,
	and $t^0_n = t^1_n$ is an element of $\overline{K_n^\ast}$, \ie
	it is in the image of~$K_n^\ast$ in the group $(\ti{n}(K))_n = K_{n}/d(K_{n+1})$,
	satisfying the conditions listed in paragraph~\ref{paragr:atom}.
	The application sending the equivalence class of
	\[
	 \tabld{x}{n}
	\]
	to the $n$-cell
	\[
	 \tableq{x}{n}{x}
	\]
	is well-defined and it is a bijection from the $n$-cells of~$\ti{n}\nu(K)$
	to the $n$-cells of~$\nu\ti{n}(K)$, which is compatible
	with source, target, identity and compositions maps.	
\end{paragr}

\begin{paragr}
	Let $K$ be a Steiner complex and consider the commutative square
	\[
	 \begin{tikzcd}
	 \lambda\nu(K) \ar[rr, "\eps_K"] \ar[d] && K \ar[d]\\
	 \ti{n}\lambda\nu(K) \ar[rr, "\ti{n}(\eps_K)"'] && \ti{n}(K)
	 \end{tikzcd}\ .
	\]
	By virtue of Theorem~\ref{thm:equivalence_Steiner} the upper horizontal
	morphism $\eps_K$ is an isomorphism and therefore the same holds for the lower
	horizontal morphism $\ti{n}(\eps_K)$. Moreover, the previous paragraph
	gives us an isomorphism $\phi\colon \ti{n}\lambda\nu(K) \to \lambda\nu\ti{n}(K)$
	and we can consider the following triangle of directed augmented complexes
	\[
	 \begin{tikzcd}
	  \ti{n}\lambda\nu(K) \ar[rr, "\phi"] \ar[rd, "\ti{n}(\eps_K)"'] &&
	  \lambda\nu\ti{n}(K) \ar[ld, "\eps_{\ti{n}(K)}"] \\
	  & \ti{n}(K)
	 \end{tikzcd}\ .
	\]
	We leave the verification of the commutativity of the above triangle as an exercise to the reader
	and we remark that it can be formally deduced by an adjunction calculus.
	As an immediate consequence, we get that the counit morphism
	$\eps_{\ti{n}(K)}\colon \lambda\nu\ti{n}(K) \to \ti{n}(K)$
	of a truncated Steiner complex is an isomorphism.
	\end{paragr}

\begin{prop}\label{remark:tronqué_fully_faithful}
 The functor $\nu\colon \Cda \to \ooCat$ is fully faithful when restricted
 to the full subcategory spanned by strong Steiner complexes and their truncations.
\end{prop}

\begin{proof}
	This follows immediately from the previous paragraph joint with Theorem~\ref{thm:equivalence_Steiner}
\end{proof}

\section{Strict \pdfn-categories and augmented directed complexes model homotopy types}

For any $n$, $1\leq n\leq \infty$, we introduce here the cosimplicial objects for
augmented directed $n$\nbd-complexes and strict $n$\nbd-categories and we show
that the former satisfies condition~(e) if and only if the latter does.

\begin{paragr}\label{paragr:nerve_Cda}
	Following Example~3.8 of~\cite{Steiner1}
	and~\cite[§5]{e}, we can define a functor
	$\cC\colon \EnsSimp \to \Cda$ as follows.	
	For any simplicial set $X$ and any integer $p\geq 0$,
	the abelian group $(\cC X)_p$ of $p$-chains of the augmented directed complex $\cC X$
	is the free abelian group with basis the set of non-degenerated $p$-simplices
	of $X$, the submonoid $(\cC X)^*_p$ of positivity of $(\cC X)_p$ is generated
	by the elements of the basis;
	for any $p$-simplex $x$ in $X$, $p>0$, we set
	\[d(x) = \sum_{=0}^p (-1)^i d^i(x)\]
	for the differential, where $d^i$ is the $i$-th face map,
	with the convention that $d^i(x) = 0$ if $d^i(x)$ is degenerate,
	and for any $0$-simplex $y$ of $X$ we set $e(y) = 1$.
	One easily checks that this indeed defines an augmented directed complex
	with basis. 
	
	The functor $\cC \colon \EnsSimp \to \Cda$ commutes with colimits
	(see Proposition~5.8 of~\cite{e}) and so it has a right adjoint
	\begin{align*}
	 \nN \colon \Cda &\to \EnsSimp\\
	 K & \mapsto \bigl(\Deltan{p} \mapsto \Hom{\Cda}(\cC(\Deltan{p}), K)\bigr)
	\end{align*}
	that we shall call the \emph{Steiner nerve}.
	We shall still call $\cC \colon \cDelta \to \Cda$
	the cosimplicial object of $\Cda$ associated to this adjoint pair.
	
	For any positive integer $n$,
	we get a functor
	$\ti{n} \cC \colon \cDelta \to \nCda{n}$ and therefore obtain
	a pair of adjoint functors
	\[\cC_n \colon \EnsSimp \to \nCda{n} \quad\text{and}\quad \nN_n \colon \nCda{n} \to \EnsSimp.\]
	The functor $\nN_n$ is the restriction of the Steiner nerve to the full subcategory $\nCda{n}$.
	We shall also denote by $\nN_\infty$ the functor $\nN \colon \Cda \to \ooCat$.
\end{paragr}
	
\begin{paragr}
	For any $n$, $1\leqslant n \leqslant \infty$, we call \ndef{weak equivalences of $\nCda{n}$}
	the morphisms of augmented directed $n$\nbd-com\-plex\-es
	whose image by $\nN_n$ is a weak equivalence of $\EnsSimp$.
\end{paragr}

\begin{paragr}
	Consider the cosimplicial object
	$\mathcal O \colon \cDelta \to \ooCat$
	defined by
	\[\On{}(\Deltan p) = \On{p} = \nu\cC(\Deltan p),\quad p\geq 0,\]
	from which we can
	define a pair of adjoint functors
	\[c_\infty \colon \EnsSimp \to \ooCat \quad\text{and}\quad N_\infty \colon \ooCat \to \EnsSimp.\]
	This cosimplicial object is isomorphic to the one introduced in~\cite{Street} by Street
	and in fact the right adjoint $N_\infty$ is known as the \ndef{Street nerve}.
	For any positive integer $n$,
	we get a functor
	$\mathcal{O}^{\leqslant n} = \ti{n} \mathcal O \colon \cDelta \to \nCat{n}$ and therefore
	a pair of adjoint functors
	\[c_n \colon \EnsSimp \to \nCat{n} \quad\text{and}\quad N_n \colon \nCat{n} \to \EnsSimp.\]
	The functor $N_n$ is the restriction of the Street nerve to the full subcategory $\nCat{n}$
	and $c_n$ is the composition of the realisation functor $c_\infty$ followed by the truncation functor $\ti{n}$.
	When $n=1$, we get the usual nerve functor, that from now on we shall simply denote by $N$.
\end{paragr}

\begin{paragr}
	For any $n$, $1\leqslant n \leqslant \infty$, we define the \ndef{weak equivalences
	of $\nCat{n}$} as the morphisms whose image by $N_n$ is
	a weak equivalence of $\EnsSimp$.
	These are also called \ndef{Thomason weak equivalences} in~\cite{nThomason}.
\end{paragr}

\begin{paragr}\label{paragr:comparison_isomorphism}
	Let $E$ be a poset. For any $n$, $1\leq n\leq \infty$, we denote by $\Onm{E}{n}$ the $n$-category $\ti{n} \nu \cC N(E)$
	and we call it the \emph{$n$-truncated oriental associated to $E$} or the \emph{$n$-truncated
	oriental of $E$}. We refer to $\nu \cC N(E)$ as the \emph{oriental
	associated to $E$} or the \emph{oriental of $E$} and we denote it by $\On{E}$.
	
	We remark that, thanks to Proposition~7.5 of~\cite{e}, the following two triangles of functors
	\[
	 \begin{tikzcd}
	  \Cda \ar[rr, "\nu"] \ar[dr, "\nN"'] &&
	  \ooCat \ar[ld, "N_\infty"] \\ &\EnsSimp
	 \end{tikzcd}
	 \qquad
	 \begin{tikzcd}
	  & \EnsSimp \ar[dl, "\cC"'] \ar[rd, "c_\infty"] \\
	  \Cda && \ooCat \rar[ll, "\lambda"]
	 \end{tikzcd}
	\]
	are commutative (up to a canonical isomorphism).
	From the adjunction
	morphism $1_{\ooCat} \to \nu\lambda$ we get a functor
	\[
	 c_\infty N(E) \to \nu\lambda c_\infty N(E) = \nu \cC N(E) = \On{E},
	\]
	which is proven to be an isomorphism in Theorem~9.5 of~\cite{e}
	(notice that by paragraph~8.4 of~\loccit~$\kappa^*(E, \xi E)$ is indeed equal to $N(E)$,
	as already observed in Remark~\ref{remark:posets_as_simplicial_complexes}).
	With a straightforward calculation one gets that the triangle
	\[
	\begin{tikzcd}
	 &N(E) \ar[dl] \ar[dr]\\
	 N_\infty c_\infty N(E) \ar[rr] && N_\infty \On{E} \cong \nN\cC N (E)
	 \end{tikzcd}
	\]
	is commutative and the horizontal morphism is an isomorphism.
	Moreover, for any $n\geq 0$, the following square of functors
	\[
	 \begin{tikzcd}
	  \Cda \ar[r, "\nu"] \ar[d, "\ti{n}"'] &
	  \ooCat \ar[d, "\ti{n}"] \\
	  \nCda{n} \ar[r, "\nu"'] &
	  \nCat{n}
	 \end{tikzcd},
	\]
	is commutative up to a canonical isomorphism 
	when restricted to augmented directed complexes with basis (see Proposition~\ref{prop:truncations}).
	Thus we have an isomorphism $\ti{n} \nu \cC(E) \cong \nu \ti{n} \cC(E)$, which gives us another canonical isomorphism
	\[
	 \ti{n}c_\infty N(E) \to \ti{n}\nu\lambda c_\infty N(E) \cong \nu\ti{n} \cC N(E).
	\]
	The naturality of $\ti{n}$ and the commutativity of the above triangle provide us with the
	following commutative triangle
	\[
	 \begin{tikzcd}
	 &N(E) \ar[dl] \ar[dr]\\
	 N_\infty \ti{n} c_\infty N(E) \ar[rr] && N_\infty \Onm{E}{n} \cong \nN \ti{n} \cC N(E)
	 \end{tikzcd},
	\]
	where the horizontal morphism is an isomorphism.
	
	Thus, it is equivalent to show condition (e) for strict $n$\nbd-categories
	and for augmented directed $n$\nbd-complexes. This result is directly inspired
	and already essentially contained in the proof of Corollary~10.13 and in
	Remark~10.14 of~\cite{e}.
\end{paragr}

\begin{thm}
For any $n$ such that $1\leqslant n \leqslant \infty$,
the functor $N_n \colon \nCat{n} \to \EnsSimp$ induces
an equivalence at the level of the homotopy categories.
In particular, strict $n$-categories model homotopy types.
\end{thm}

\begin{proof}
By Theorem~\ref{thm:illusie-quillen}, it is enough to show that the cosimplicial objects
$\On{}^{\leq n}$ of $\nCat{n}$ satisfy condition (e)
for any $n$, $1\leqslant n \leqslant \infty$.
This is precisely Corollary~10.11 of~\cite{e}, but we shall give an independent
proof in the next section (see Theorem~\ref{thm:e}).
\end{proof}

\begin{thm}
For any $n$ such that $1\leqslant n \leqslant \infty$,
the functor $\nN_n \colon \nCda{n} \to \EnsSimp$ induces an
equivalence at the level of the homotopy categories.
In particular, augmented directed $n$-complexes model homotopy types.
\end{thm}

\begin{proof}
 By Theorem~\ref{thm:illusie-quillen}, it is enough to show that
 the cosimplicial objects $\cC_n$ of $\nCat{n}$ satisfies condition~(e)
 for any $n$, $1\leqslant n \leqslant \infty$.
 This follows by Paragraph~\ref{paragr:comparison_isomorphism}
 and Corollary~10.11 of~\cite{e}, but we shall give an
 independent proof in the next section (see Theorem~\ref{thm:e}).
\end{proof}

\begin{remark}\label{remark:oonCat}
	For any $n$ and $k$, $0 \leq k \leq n \leq \infty$, consider the full subcategory
	$\nCat{(n, k)}$ of $\nCat{n}$ spanned by the strict $n$\nbd-categories whose
	$i$\nbd-cells are all (strictly) invertible for $i > k$.
	The embedding functor $\nCat{(n, k)} \to \nCat{n}$ has a left adjoint,
	that we will denote by $L_k$, so that we can consider a cosimplicial
	object
	\[ L_k \,\On{}^{\leq n} \colon \cDelta \to \nCat{(n, k)}. \]
	For $k \ge 1$, this functor also satisfies condition (e) and we will sketch a
	proof of this result at the end of the next section (see Remark~\ref{remark:oonCat_proof}).
\end{remark}

\section{The simplicial homotopy}

\begin{paragr}
The aim of this section is to establish condition (e)
for the cosimplicial object $\cC_n \colon \Delta \to \nCda{n}$
of augmented directed $n$\nbd-complexes
and for the cosimplicial object
$\Onm{}{n}\colon \Delta \to \nCat{n}$
of strict $n$-categories, for $1 \leq n \leq \infty$.
More explicitly, we shall show that the canonical morphisms
$N(E) \to \nN_n\cC_n N(E)$ and $N(E) \to N_n c_n N(E)$ are
simplicial weak equivalences, for any $1\leq n \leq \infty$.
As the computations of paragraph~\ref{paragr:comparison_isomorphism}
show, the former is a simplicial weak equivalence if and only if
the latter is so.
\end{paragr}

\emph{We fix for this section a poset $E$.}

\begin{paragr}\label{paragr:chains_oriental_poset}
	The augmented directed complex $\cC N(E)$ can be described as follows (cf.~paragraph~\ref{paragr:nerve_Cda}).
	For $p\geq 0$, the non-degenerated $p$-simplices of $N(E)$ are strictly
	increasing maps $\Deltan{p} \to E$.
	Thus, for $p\geq 0$ the elements of $(\cC N(E))_p$ (resp. $(\cC N(E))^*_p$) can be
	identified with the abelian group (resp. abelian monoid) freely generated
	by the family of $(p+1)$-tuples of the form
	\[
	 (i_0, i_1, \dots, i_p)\,,\quad i_0 < i_1 < \dots < i_p\,.
	\]
	The differential is defined by
	\[
	 d(i_0, i_1,\dots, i_p) = \sum_{k=0}^p (-1)^k (i_0, i_1, \dots, \widehat{i_k}, \dots, i_p)\,,\quad p>0\,,
	\]
	where $(i_0, \dots, \widehat{i_k}, \dots, i_p) = (i_0, \dots, i_{k-1}, i_{k+1}, \dots, i_p)$,
	and the augmentation by $e(i_0) = 1$. Theorem~8.6 of~\cite{e} shows that $\cC N(E)$ is
	a strong Steiner complex.
\end{paragr}

\emph{For the remaining of the section, we fix an $n$ such that $1 \le n \le \infty$
and we denote the identity functor of $\ooCat$ by $\ti{\infty}$ in order to give a unified treatment.}

\begin{paragr}\label{paragr:juxtaposition}
	Let $p$ be an integer, $p\geq 0$ and consider two
	morphisms $f$ and $g$ from $\cC\Deltan{p}$ to
	$\ti{n}\cC N(E)$ such that $f(i)\leq g(i)$ for
	any $i = 0, \dots, p$. Fix an element $(j_0, \dots, j_m)$
	of the basis of $\cC \Deltan{p}$ with $m\ge 0$, an integer $\ell$,
	$0\leq \ell \leq m-1$. If $m\le n$, consider
	\begin{align*}
	 f(j_0, \dots, j_\ell) = \sum_{i \in I} (y_{0, i}, \dots, y_{\ell, i})
	 &&\mathrm{and}&&
	 g(j_{\ell+1}, \dots, j_m) = \sum_{k \in K} (y_{\ell+1, k}, \dots, y_{m , k})
	\end{align*}
	be the unique forms as sums of elements of the canonical basis of the free abelian groups
	$\cC N(E)_\ell$ and $\cC N(E)_{m-\ell-1}$ and set
	\begin{align*}
	 H_\ell(j_0, \dots, j_m) &= f(j_0, \dots, j_\ell)g(j_{\ell+1}, \dots, j_m)\\
	 	&= \sum_{(i, k) \in I \times K}\overline{(y_{0, i}, \dots, y_{\ell, i}, y_{\ell+1, k}, \dots, y_{m, k})} .
	\end{align*}
	If $m> n$, then we set $H_\ell(j_0, \dots, j_m) = 0$.
	In order for the above sum to be well-defined, we have to check that $y_{\ell, i} \leq y_{\ell+1, k}$
	for all $i$ in $I$ and $k$ in $K$. The following lemma immediately implies the inequalities
	\[
	 y_{\ell, i} \leq f(j_\ell) \leq g(j_{\ell + 1}) \leq y_{\ell+1, k}.
	\]
	\end{paragr}
	
	\begin{lemma}\label{lemma:juxtaposition_well-defined}
		Let $p$ be an integer, $p\geq 0$, and consider a morphism $f$ from $\cC \Deltan{p}$ to
		$\cC N(E)$. Then, any element $(y_0, y_1, \dots, y_p)$ of the basis of $\cC N(E)$
		appearing in the support of $f(0, 1, \dots, p)$ is such that $f(0) \leq y_0$ and $y_p \leq f(p)$.
	\end{lemma}
	
	\begin{proof}
		By virtue of Lemma~8.7 of~\cite{e} we have that
		\[
		 \atom{(0, 1, \dots, p)}^0_1 = (0, p)\quad\text{and}\quad \atom{(0, 1, \dots, p)}^1_1 = (0, 1) + (1, 2) + \dots + (p-1, p)
		\]
		and by the linearity of $f$ we get
		\[
		f\bigl(\atom{(0, 1, \dots, p)}^0_1\bigr) = f(0, p)
		\]
		and
		\[
		f\bigl(\atom{(0, 1, \dots, p)}^1_1\bigr) = f(0, 1) + f(1, 2) + \dots + f(p-1, p).
		\]
		For any $0\leq i < p$, the chain $f(i, i+1)$ of $\cC N(E)$ is simply a sum
		\[
		 f(i, i+1)  = (x_{0, i}, x_{1, i}) + (x_{1, i}, x_{2, i}) + \dots + (x_{n_i-1, i}, x_{n_i, i})
		\]
		of $1$\nbd-chains of $\cC N(E)$, where $x_{0, i} = f(i)$ and $x_{n_i, i} = f(i+1)$ (cf.~paragraph~10.3 of~\cite{e}). Similarly,
		\[
		f(0, p) = (x_0, x_1) + (x_1, x_2) + \dots + (x_{m-1}, x_m),
		\]
		where $x_0 = f(0)$ and $x_m = f(p)$.
		
		Setting $S' = f(0, p)$ and
		$S = f(0, 1) + \dots + f(p-1, p)$, Proposition~10.6 of~\cite{e} states that
		\[
		 \Hom_{\On{F}}(S', S) = \Hom_{\On{E}} (S', S), 
		\]
		where $F$ is the linearly ordered subset of $E$ with elements the
		$x_{k, i}$, for all $0\leq i < p$ and $0\leq k \leq n_i$.
		Hence any element $(y_0, y_1, \dots, y_p)$
		appearing in the support of $f(0, 1, \dots, p)$
		is a $p$\nbd-chain of the augmented directed complex $\cC N(F)$.
		The explicit description of the chains of $\cC N(F)$ given in paragraph~\ref{paragr:chains_oriental_poset}
		immediately implies that $f(0) \leq y_0$ and $y_p \leq f(p)$.
	\end{proof}
	
	\begin{paragr}\label{paragr:juxtaposition2}
		With the same notations as in paragraph~\ref{paragr:juxtaposition},
	notice that we have the following relation
	\begin{align*}
	 d\bigl(H_\ell(j_0, \dots, j_m)\bigr) &= \sum_{i, k} d(y_{0, i}, \dots, y_{\ell, i}, y_{\ell+1, k}, \dots, y_{m, k}) \\
	 &= \sum_{i, k} \sum_{s=0}^\ell (-1)^s (y_{0, i}, \dots, \widehat{y_{s, i}}, \dots, y_{\ell, i}, y_{\ell+1, k}, \dots, y_{m, k}) \\
	 &\phantom{=} + \sum_{i, k} \sum_{s=\ell + 1}^m (-1)^s (y_{0, i}, \dots, y_{\ell, i}, y_{\ell+1, k}, \dots, \widehat{y_{s, k}}, \dots, y_{m, k})\\
	 &= d\bigl(f(j_0, \dots, j_\ell)\bigr)g(j_{\ell +1}, \dots, j_m) \\
	 &\phantom{=} + (-1)^{\ell + 1} f(j_0, \dots, j_\ell)d\bigl(g(j_{\ell +1}, \dots , m)\bigr) \\
	 &= d\bigl(H_\ell(j_0, \dots, j_\ell)\bigr)H_\ell(j_{\ell + 1}, \dots, j_m)\\
	 &\phantom{=} + (-1)^{\ell+1} H_\ell(j_0, \dots, j_\ell)d\bigl(H_\ell(j_{\ell +1}, \dots, j_m)\bigr) .
	\end{align*}
\end{paragr}

\begin{paragr}\label{paragr:definition_homotopy}
	Let $f$ and $g$ be two endomorphisms of $N_\infty(\Onm{E}{n})$ and suppose that
	$f(i) \leq g(i)$ for any $i$ in $E$, \ie for any 0-simplex of $N_\infty(\Onm{E}{n})$.
	Let $p\geq 0$, let $\phi\colon \Deltan{p} \to \Deltan{1}$ be a morphism and let $x \colon \On{p} \to \Onm{E}{n}$
	be an \oo-functor. By virtue of Remark~\ref{remark:tronqué_fully_faithful}, we define a \oo-functor
	\[
	 H(\phi, x) \colon \On{p} \to \Onm{E}{n}
	\]
	by means of a morphism of augmented directed complexes from $\cC\Deltan{p}$ to $\ti{n}\cC  N (E)$,
	that we shall still call $H(\phi, x)$. 
	For any element $(j_0, \dots, j_m)$ of the basis of
	$\cC\Deltan{p}$ there exists a unique integer $k_\phi$
	such that $-1 \leq k_\phi \leq m$ and such that
	\[
	 \phi(j_k) = \begin{cases}
	              0, & \mbox{if } 0\leq k \leq k_\phi,\\
	              1, & \mbox{if } k_\phi +1 \leq k \leq m.
	             \end{cases} 
	\]
	If $m > n$, we have that $H(\phi, x)(j_0, \dots, j_m)$ must be zero.
	If $m \le n$ we set
	\[
	 H(\phi, x)(j_0, \dots, j_m) = \bigl(\lambda(f(x))(j_0, \dots, j_{k_\phi})\bigr)\bigl(\lambda(g(x))(j_{k_\phi+1}, \dots, j_m)\bigr)
	\]
	and we notice that
	\[ \lambda(f(x))(j_{k_\phi}) = f(x_{j_{k_\phi}}) \le g(x_{j_{k_\phi}}) \le g(x_{j_{k_\phi+1}}) = \lambda(g(x))(j_{k_\phi+1}) \]
	so that by Lemma~\ref{lemma:juxtaposition_well-defined} $H(\phi, x)(j_0, \dots, j_m)$ is a
	well-defined function from the basis of $\cC \Deltan{p}$ to $\ti{n}\cC  N(E)$
	(cf.~also paragraph~\ref{paragr:juxtaposition}).
	We also remark that the only cases in which we have to deal with
	elements of $(\ti{n}\cC  N(E))_n$ without a prescribed choice for
	the equivalent class are precisely when $m=n$ and $k_\phi = -1$ or $k_\phi = n$,
	for which we get respectively
	\[
	 H(\phi, x)(j_0, \dots, j_n) = \lambda(g(x))(j_0, \dots, j_n)
	\]
	and
	\[
	 H(\phi, x)(j_0, \dots, j_n) = \lambda(f(x))(j_0, \dots, j_n)\,.
	\]
\end{paragr}

\begin{paragr}
	We have to check that $H(\phi, x)$ is indeed a morphism of augmented directed complexes.
	The compatibility with the augmentation and the positivity submonoids is clear.
	Let us verify the compatibility with the differential.
	Let $(j_0, \dots, j_m)$ be an element of the basis of $\cC\Deltan{p}$. If $m > n$ or
	 $m = n$ and $k_\phi = -1$ or $k_\phi= n$ the verification is trivial.
	In the other cases we have
	\begin{align*}
		H(\phi, x)(d(j_0, \dots, j_m))
		&= H(\phi, x)\bigl(d(j_0, \dots, j_{k_\phi})(j_{k_\phi+1}, \dots, j_m)\bigr)\\
		 &\phantom{= } +(-1)^{k_\phi+1} H(\phi, x)\bigl((j_0, \dots, j_{k_\phi})d(j_{k_\phi+1}, \dots, j_m)\bigr) \\
		&= \Bigl(\lambda(f(x))\bigl(d(j_0, \dots, j_{k_\phi})\bigr)\Bigr)\Bigl( \lambda(g(x))(j_{k_\phi+1}, \dots, j_m)\Bigr)\\
		 &\phantom{=} + (-1)^{k_\phi+1} \Bigl(\lambda(f(x))(j_0, \dots, j_{k_\phi})\Bigr)\Bigl( \lambda(g(x))\big(d(j_{k_\phi+1}, \dots, j_m)\bigr)\Bigr) \\
		&= \Bigl(d\bigl(\lambda(f(x))(j_0, \dots, j_{k_\phi})\bigr)\Bigr)\Bigl(\lambda(g(x))(j_{k_\phi+1}, \dots, j_m)\Bigr)\\
		 &\phantom{=} + (-1)^{k_\phi+1} \Bigl(\lambda(f(x))(j_0, \dots, j_{k_\phi})\Bigr)\Bigl( d\bigl(\lambda(g(x))(j_{k_\phi+1}, \dots, j_m)\bigr)\Bigr) \\
		&= \Bigl(d\bigl(H(\phi, x)(j_0, \dots, j_{k_\phi})\bigr)\Bigr)\bigl(H(\phi, x)(j_{k_\phi+1}, \dots, j_m)\bigr)\\
		 &\phantom{=} +(-1)^{k_\phi+1} \bigl(H(\phi, x)(j_0\, \dots, j_{k_\phi})\bigr)\Bigl(d\bigl(H(\phi, x)(j_{k_\phi+1}, \dots, j_m)\bigr)\Bigr) \\
		&= d\bigl(H(\phi, x)(j_0, \dots, j_m)\bigr) ,
	\end{align*}
	where the last equality follows from paragraph~\ref{paragr:juxtaposition2}.
\end{paragr}

\begin{paragr}
	We have to check that $H$ is a simplicial homotopy from $f$ to $g$,
	that is to say a morphism of simplicial sets $\Deltan{1} \times N_\infty(\Onm{E}{n}) \to N_\infty(\Onm{E}{n})$
	such that the restriction to $\{0\} \times N_\infty(\Onm{E}{n})$
	(resp. the restriction to~$\{1\}\times N_\infty(\Onm{E}{n})$)
	gives~$f$ (resp. gives~$g$).
	For any \oo-functor $x \colon \On{p} \to \Onm{E}{n}$,
	the restriction to $\{0\} \times N_\infty(\Onm{E}{n})$ of $H$ gives
	\[
	 H(\phi, x) = \lambda(f(x)) 
	\]
	by definition, where $\phi$ is the unique morphism $\Delta_p \to \Delta_0$,
	and similarly for the restriction to $\{1\}\times N_\infty(\Onm{E}{n})$.
	
	We have to verify that for any
	morphisms $\psi\colon \Deltan{p'} \to \Deltan{p}$ and $\phi\colon \Deltan{p} \to \Deltan{1}$
	and any \oo-functor $x \colon \On{p} \to \Onm{E}{n}$, the relation
	\[
	 H(\phi, x) \On{\psi} = H(\phi\psi, x\On{\psi})
	\]
	holds. We shall show this equality using augmented directed complexes.
	
	Let $(j_0, \dots, j_m)$ be an element of the basis of $\cC\Deltan{p'}$
	and let us set $k' = k_{\phi\psi}$. If $m> n$ the equality above holds trivially as well as
	the cases $m=n$ and $k'=-1$ or $k'= n$.
	In the other cases we have
	\begin{align*}
		\bigl(H(\phi, x) \cC\psi\bigr)(j_0, \dots, j_m)
		&= H(\phi, x)\bigl(\psi(j_0), \dots, \psi(j_m)\bigr) \\
		&= \bigl(\lambda(f(x))(\psi(j_0), \dots, \psi(j_{k'})\bigr)
		\bigl(\lambda(g(x))(\psi(j_{k'+1}), \dots, \psi(j_m))\bigr)\\
		&=  \bigl(\lambda(f(x))\cC\psi(j_0, \dots, j_{k'})\bigr)
		\bigl(\lambda(g(x))\cC\psi(j_{k'+1}, \dots, j_m)\bigr)\\
		&= H(\phi\psi, x\On{\psi})(j_0, \dots, j_m)\ ,
	\end{align*}
	where the last equality holds since  we have the relation
	$f(x\On{\psi}) = f(x)\On{\psi}$ by naturality of $f$
	and the functoriality of $\lambda$ gives
	$\lambda(f(x)\On{\psi}) = \lambda(f(x))\cC\psi$.

\end{paragr}

\begin{thm}\label{thm:e}
	For any $n$ such that $1\leqslant n \leqslant \infty$,
	the cosimplicial objects
	\[
	\cC_n\colon \cDelta \to \nCda{n}
	\quad\text{and}\quad
	c_n\colon \cDelta \to \nCat{n}
	\]
	both satisfy condition~(e).
\end{thm}

\begin{proof}
      In order to show that the morphism $\etan{E}{n}\colon N(E) \to N_\infty(\Onm{E}{n})$
      is a simplicial weak equivalence for any $n$ such that $1\le n\le \infty$,
      and thus also the morphism
      $N(E) \to N_nc_nN(E)$ is so by paragraph~\ref{paragr:comparison_isomorphism},
      we consider the canonical \oo-functor $\epsn{E}{n}\colon \Onm{E}{n} \to E$.
      One immediately checks that the composition $N_\infty(\epsn{E}{n}) \etan{E}{n}$
      is the identity on $N(E)$; indeed, for any $p\ge 0$ and any functor
      $x\colon\Deltan{p} \to E$ and \oo-functor $y\colon \On{p} \to \Onm{E}{n}$, we have
      \[
       \etan{E}{n} (x) = \mathcal O^{\le n}(x)
       \quad\text{and}\quad
       \epsn{E}{n}\circ y = \{y(0) \le y(1) \le \dots \le y(p)\}\,.
      \]
      Furthermore, for any object $i$ of $E$, \ie for any $0$-simplex of
      $\Onm{E}{n}$, we have
      \[
       \etan{E}{n} N_\infty(\epsn{E}{n})(i) = \etan{E}{n}(i) = i
      \]
      and therefore we can set $f = \etan{E}{n} N_\infty(\epsn{E}{n})$
      and $g = \id{\Onm{E}{n}}$ and obtain a simplicial homotopy
      $H$ from $f$ to $g$, as defined in paragraph~\ref{paragr:definition_homotopy}.
      Finally, the morphism $N_\infty(\epsn{E}{n})$ is a deformation retract and hence
      it is a simplicial weak homotopy equivalence.
\end{proof}

\begin{remark}\label{remark:oonCat_proof}
	One can show that strict $(n, k)$\nbd-categories model homotopy types,
	for all $1\leq k\leq n \leq \infty$.
	Indeed, the cosimplicial object $L_k \Onm{}{n} \colon \cDelta \to \nCat{(n, k)}$
	introduced in Remark~\ref{remark:oonCat} satisfies condition~(e).
	We sketch the argument of the proof, which is
	inspired by the strategy used in~\cite{e} to establish condition~(e) for strict $n$\nbd-categories.
	
	We want to show that the \oo-functor $\epsilon_E\colon L_k \Onm{E}{n} \to E$ is a weak equivalence
	for any poset~$E$.
	One checks that the functor $L_k$ from $\nCat{n}$ to $\nCat{(n, k)}$ is the identity
	on objects and $1$\nbd-cells.
	In particular, the functor $\epsilon_E$ is a bijection on objects (see~\ref{paragr:chains_oriental_poset})
	and we can therefore apply Theorem~7.10 of~\cite{e}, which states that in order
	to prove that $\epsilon_E$ is a weak equivalence, it suffices to verify
	that for any element $x$ and $y$ of $E$, the \oo-functor
	\[
	 \Hom_{L_k \Onm{E}{n}}(x, y) \to \Hom_{E}(x, y)
	\]
	is a weak equivalence. As the functor $L_k$ is the identity on the $1$\nbd-cells,
	both the source and the target of this \oo-functor are empty if $x \not \leq y$.
	Suppose that $x \leq y$. Since $E$ is a poset, we have that $\Hom_{E}(x, y)$
	is the terminal \oo-category. We are left to show that $\Hom_{L_k \Onm{E}{n}}(x, y)$
	is weakly contractible.
	
	We define $F$ to be the subposet of~$E$ whose elements are the $z$ in $E$ such that
	$x \leq z \leq y$. We have $\Hom_{\On{F}}(x, y) = \Hom_{\On{E}}(x, y)$.
	One can check that the functor $L_k\ti{n}$ commutes with finite products for any $0\leq k\leq n\leq \infty$.
	This implies that for any~$k$ and~$n$ such that $1\leq k\leq n\leq \infty$,
	any \oo-category~$A$ and any pair of objects~$a$ and~$b$ of~$A$,
	we have a canonical isomorphism
	\[
	 L_{k-1}\ti{n-1}\Hom_A(a, b) \cong \Hom_{L_k\ti{n}A}(a, b).
	\]
	In particular, the following relation holds
	\begin{equation*}
	 \begin{split}
	 \Hom_{L_k \Onm{E}{n}}(x, y) &\cong L_{k-1}\ti{n-1}\Hom_{\On{E}}(x, y)\\
	 &= L_{k-1}\ti{n-1}\Hom_{\On{F}}(x, y)\\ &\cong \Hom_{L_k \Onm{F}{n}}(x, y).
	 \end{split}
	\end{equation*}
	Therefore, it is enough to show that for any poset~$F$ with smallest element~$x$,
	the \oo-category $\Hom_{L_k \Onm{F}{n}}(x, y)$ is weakly contractible.
	
	A slight modification of Theorem~6.6 of~\cite{e} shows that for any poset $F$ with smallest element~$x$,
	one can define a contraction of augmented directed complexes of $\cC N(F)$ (see paragraph~4.6 of~\loccit).
	By~Proposition~B.18 of~\loccit, this gives a contraction of the \oo-category $\On{F}$
	with centre $x$, \ie an oplax transformation from the constant \oo-functor on $x$
	to the identity \oo-functor on $\On{F}$
	satisfying an additional property (see paragraphs~B.9 of~\loccit).
	One can show that the universal property of the \oo-functor $\On{F} \to L_k \Onm{F}{n}$
	implies that this contraction induces a contraction on $L_k\Onm{F}{n}$ with centre $x$.
	From the results of Appendix~B of~\loccit one finally gets that
	the \oo-category $\Hom_{L_k\Onm{F}{n}}(x, y)$ is weakly contractible for any element $y$ of $F$.

	A detailed proof of this result will appear in the Ph.D.~thesis
	of the author.
\end{remark}

\printbibliography
\end{document}